\documentclass[a4paper,12pt]{amsart}
\usepackage[top=1.5in, bottom=1in, left=0.9in, right=0.9in]{geometry}

\usepackage[utf8]{inputenc}
\usepackage[all]{xy}
\usepackage{color}
\usepackage{hyperref}
\usepackage{enumitem}

\usepackage{amsmath, mathtools}
\usepackage{hyperref}
\hypersetup{colorlinks=true}
\usepackage{amssymb}
\usepackage{amsfonts}
\usepackage{graphicx}
\usepackage{mathrsfs}
\usepackage{amscd}
\usepackage[all]{xy}
\usepackage{hyperref}
\usepackage{amsthm}
\usepackage{cleveref}
\usepackage{verbatim}
\usepackage{amscd}
\usepackage{mathtools}
\usepackage{comment}

\newtheorem{thm}{\bf Theorem}[section]

\newtheorem{prop}[thm]{\bf Proposition}
\newtheorem{lemma}[thm]{\bf Lemma}
\newtheorem{cor}[thm]{\bf Corollary}
\theoremstyle{definition}
\newtheorem{definition}[thm]{\bf Definition}
\theoremstyle{remark}
\newtheorem{remark}[thm]{\bf Remark}

\newtheorem{question}[thm]{\bf Question}
\newtheorem{example}[thm]{\bf Example}
\numberwithin{equation}{section}

\DeclareMathOperator{\depth}{{depth}}
\DeclareMathOperator{\grade}{{grade }}

\DeclareMathOperator{\Supp}{{Supp}}
\DeclareMathOperator{\Ass}{{Ass}}

\DeclareMathOperator{\Tor}{{Tor}}

\DeclareMathOperator{\ord}{{ord}}
\DeclareMathOperator{\len}{{length}}

\def\f0{\mathbf{0}}

\def\fp{\mathbf{p}}

\def\fm{\mathfrak{m}}

\def\fp{\mathfrak{p}}

\newcommand{\ses}[3]{0 \to {#1} \to {#2} \to {#3} \to 0}

\begin{document}

\title[On colon operations]{On colon operations and special types of ideals}

\author[Hailong Dao]{Hailong Dao}
\address{Hailong Dao\\ Department of Mathematics \\ University of Kansas\\405 Snow Hall, 1460 Jayhawk Blvd.\\ Lawrence, KS 66045}
\email{hdao@ku.edu}

 \begin{abstract}
 We record a general asymptotic formula for colon of ideals and proceed to give some applications  regarding $\fm$-full, weakly $\fm$-full, and full ideals.  
  
\end{abstract}
\subjclass[2010]{}

\maketitle

\section{Introduction}
Let $R$ be a commutative ring, $I$ an ideal and $M\subseteq N$ finitely generated modules. One purpose of this note is to record a formula of the form $$I^nM:_NJ = 0:_NJ+I^{n-1}M$$ for $n\gg 0$ and an ideal $J\subseteq I$ containing any ``general" generator of $I$. The precise and key local statement is Theorem \ref{mainthm}. 
We then give some corollaries. For instance if $I,J,K$ are $R$-ideals, then there is a number $t>0$ such that 
$$(J+I^nK):I = J:I +I^{n-1}K$$
for each $n\geq t$. See Corollaries \ref{maincor} and \ref{coloncor}. 

Although these results are not too hard to prove and some special forms of them are well-known to experts (see Remark \ref{BrR} and \cite{Br, HS, R}), we could not locate the most general versions  in the literature and found them rather convenient, thus it seems worth writing down.

Our main application (and original motivation) is to study the properties  $\{$$\fm$-full, weakly $\fm$-full, full$\}$ for ideals {\it asymptotically} in a local ring $(R,\fm)$. These type of ideals have recently attracted renewed attention for some remarkable homological properties, see Definition \ref{fulldef} and Remark \ref{fullref}. Using the results on colons we can quickly show that if $(P)$ is one of these properties, and $I$ satisfies $(P)$, then $I+K\fm^n$ is $(P)$ for $n\gg0$, see Theorem \ref{fullthm}. 

In the final section we deal with regular local rings of dimension two and give stronger versions of  the previous results. For instance, we give a precise condition for when $(I+J):x =I:x+J:x$ for a general $x\in \fm$ (Proposition \ref{dim2colon}) and apply it to show when the sum of two $\fm$-full ideals is $\fm$-full (Corollary \ref{dim2cor}). We show that certain invariants defined using the properties $\{$$\fm$-full, weakly $\fm$-full, full$\}$ and our stabilizing results coincide in this case. \\

\noindent\textbf{Acknowledgements}: We thank Neil Epstein and Craig Huneke for helpful comments and encouragements. We thank the anonymous referee who pointed out a mistake in our original proof of Lemma \ref{mainlem}. 

\section{General results on colon}

First we consider the local situation $(R,\fm,k)$. We say that $x\in I$ is a {\it general} element if the image of $x$ in $V=I/\fm I$ lies in a non-zero Zariski open subset $U$ of $V$.

\begin{lemma}\label{mainlem}
Let $(R,\fm,k)$ be a local ring with infinite residue field. Let $I$ be an ideal of $R$ and $M\subseteq N$ be finitely generated $R$-modules. Assume that $\grade(I,N)>0$. Then there is a number $t>0$ such that for a general element $x$ in $I$, we have

$$I^nM:_Nx = I^{n-1}M $$
for each $n\geq t$. 
\end{lemma}

\begin{proof}
The case $M=N$ is \cite[Lemma 8.5.3]{HS}. Thus, it is enough to show that for $n\gg0$, $I^nM:_Nx\subseteq M$, as then we would have $$I^nM:_Nx = I^nM:_Mx = I^{n-1}M$$

Let $\mathcal S$ be the set of associated primes of $N/M$ that are not in $\Supp(R/I)$. We can choose $x$ outside all primes in $\mathcal S$. Let $L=M:_Nx^{\infty}$, then $L/M$ is $I$-torsion. Clearly $I^nM:_Nx\subseteq L$. Also, by the case $M=N$, $I^nM:_Nx\subseteq I^{n-1}N$, so $I^nM:_Nx\subseteq L\cap  I^{n-1}N$. But by Artin-Rees Lemma, there is a constant $c$ so that if $n-1\geq c$
$$L\cap I^{n-1}N = I^{n-1-c}(L\cap I^cN)\subseteq I^{n-1-c}L \subseteq M$$
if $n-c$ is big enough.  
\end{proof}

\begin{thm}\label{mainthm}
Let $(R,\fm,k)$ be a local ring with infinite residue field. Let $I$ be an ideal of $R$ and $M\subseteq N$ be finitely generated $R$-modules. Then there is a number $t>0$ such that for a general element $x$ in $I$, given any ideal $J$ with $x\in J \subseteq I$ we have 

$$I^nM:_NJ = 0:_NJ+I^{n-1}M$$
for each $n\geq t$. 
\end{thm}

\begin{proof}
Let $L=\Gamma_I(N)= 0:_NI^{\infty}$,  $N'=N/L$ and $M'=(M+L)/L$ (which can be zero modules). Then $L:_NI=L$, in other words $\grade(I,N')>0$, so we can choose $t_1$ such that for any superficial element $x$ (which is also regular on $N$) in $I$ with respect to $N$, the following holds:  $I^nM':_{N'}x=I^{n-1}M'$ for $n\geq t_1$, by Lemma  \ref{mainlem}. 

On the other hand, by Artin-Rees Lemma there is $t_2$ such that $L\cap I^nM=0$ for $n\geq t_2$. Choose $t=\max\{t_1,t_2\}$. For any $n\geq t$
$$0:_NJ = (L:_NJ) \cap  (I^nM:_NJ) = L\cap (I^nM:_NJ)$$
(the second equality holds since $J$ contains a regular element on $N'$, hence $L:_NJ=L$). 

We rewrite $I^nM':_{N'}x=I^{n-1}M'$  as $$(L+I^nM):_Nx =L+I^{n-1}M$$ Let $u\in I^nM:_NJ \subseteq (L+I^nM):_Mx$.  Thus $u = v+w$ with $v\in L$ and $w\in I^{n-1}M$. But then $v \in L\cap I^nM:_NJ = 0:_NJ$, which gives the non-trivial inclusion and proves the desired equality. 
\end{proof}

\begin{remark}
Looking at the proof, one sees that the only place we use $J\subseteq I$ is to show that $v= u-w \in I^nM:_NJ$. So it is enough to assume that $JI^{n-1}M\subseteq I^nM$ for $n\gg0$, in other words $J$ is inside the Ratliff-Rush closure of $I$ (with respect to $M$) (see \cite{N, RR}).  
\end{remark}

\begin{cor}\label{maincor}
Let $R$ be a Noetherian commutative ring. Let $I$ be an ideal of $R$ and $M\subseteq N$ be finitely generated $R$-modules. Then there is a number $t>0$ such that 
$$I^nM:_NI = 0:_NI+I^{n-1}M$$
for each $n\geq t$. 

\end{cor}
\begin{proof}
Let $X_n, Y_n$ be the left and right hand sides respectively. Clearly $X_n \supseteq Y_n$, so to prove the equality it is enough to prove $(X_n)_{\fp} = (Y_n)_{\fp}$ for each $\fp \in \Ass(Y_n)$. As $S = \cup_{i\geq 1}Y_n$ is finite, see \cite{Br}, we can reduce to the local case (our $t$ will be the maximal of all $t_{\fp}$ that works for each localization at $\fp \in S$. 
Once reduced to the case  $(R,\fm,k)$ local  we can make a faithfully flat extension to assume $k$ is infinite and apply \ref{mainthm} with $J=I$ (note that as $k$ is infinite, general elements exist). 
\end{proof}

\begin{remark}\label{BrR}
The case $M=N$ and $0:_MI=0$ of Corollary \ref{maincor} appeared as Lemma (4) in \cite{Br}, which refers to the proof of \cite[Theorem 4.1]{R}, which was the case when $M=N=R$ and $0:I=0$.
\end{remark}

\begin{cor}\label{coloncor}
Let $R$ be a Noetherian commutative ring and $I,J,K$ be $R$-ideals. Then there is a number $t>0$ such that 
$$(J+I^nK):I = J:I +I^{n-1}K$$
for each $n\geq t$. 

\end{cor}
\begin{proof}
We  apply \ref{maincor} with $M=(K+J)/J, N=R/J$. 
\end{proof}

\section{Applications}

For an ideal $I$ in a local ring $(R,\fm,k)$ let $\mu(I)$ denote the minimal number of generators of $I$ and $\ord(I) = \max \{t \ |  I\subseteq m^t\}$. 

\begin{definition}\label{fulldef} Let $(R,\fm,k)$ be a local ring. We say that an ideal $I$ of $R$ is 
\begin{enumerate}
\item $\fm$-full if $I\fm:x =I$ for a general $x\in \fm$ (assuming $k$ is infinite). 
\item full if $I:x =I:\fm$ for a general $x\in \fm$ (assuming $k$ is infinite). 
\item weakly $\fm$-full (or basically full) if $I\fm:\fm=I$. 
\item Burch if $I\fm:\fm \neq I:\fm$ (equivalently $I\fm \neq (I:\fm)\fm$). 
\end{enumerate}
\end{definition}

\begin{remark}\label{fullref}
The above types of ideals have been studied by many authors and shown to enjoy remarkable properties. When $\depth R/I=0$, we have $\fm$-{full} $\implies$ weakly $\fm$-full $\implies$ Burch. Burch ideals and their quotients enjoys unexpectedly strong properties (\cite{DKT}).  Weakly $\fm$-full ideals are also called basically full in \cite{HRR} and weakly $\fm$-full in \cite{CIST}. See \cite{B,CIST, DKT, HLNR, HRR, HW,  W, W1, W2} and the references therein for more details. 
\end{remark}

\begin{remark}
Even when $k$ is finite, we can still define $\fm$-fullness or fullness by passing to the faithfully flat extension $S_{\fm S}$ with $S=R[X_1,\dots, X_n]$, $n=\mu(\fm)$. 
\end{remark}

\begin{thm}\label{fullthm}
Let $(R,\fm,k)$ be a local ring and $J, K$  ideals of $R$. Let $(P)$ be one of the properties $\{$$\fm$ -full, weakly $\fm$-full, full$\}$. The following are equivalent 
\begin{enumerate}
\item $J$ is $(P)$. 
\item $J+K\fm^n$ is $(P)$ for $n\gg0$. 
\end{enumerate}
\end{thm}

\begin{proof}
We shall give the proof for $(P)=``\fm$-full", the other cases are similar. By \ref{mainthm}, we have for $J_n=J+\fm^n$ and $n\gg0$:
$$J_n\fm:x = J\fm:x + K\fm^{n} $$
So $J_n$ is $\fm$-full for $n\gg0$ is equivalent to $J\fm:x +K\fm^{n} = J+K\fm^n$ for $n\gg0$. Working in $R/J$, this is equivalent to $\overline{J\fm:x} \subseteq \overline{K\fm^{n}}$ for $n\gg0$,  which is  equivalent to $\overline {J\fm:x} =0$, or  $J\fm:x=J$. 
\end{proof}

The below Corollary extends \cite[Theorem 7.2]{HRR}, see Remark \ref{fullref} and \cite[Proposition 3.3 (ii)]{W2}.

\begin{cor}\label{fullcor}
Let $(R,\fm,k)$ be a local ring. The following are equivalent:
\begin{enumerate}
\item $\depth R>0$.
\item $K\fm^n$ is  weakly $\fm$-full for  $n\gg0$ and some ideal $K$. 
\item $K\fm^n$ is weakly $\fm$-full for  $n\gg0$ and any ideal $K$.
\item $K\fm^n$ is  $\fm$-full for  $n\gg0$ and some ideal $K$.
\item $K\fm^n$ is  $\fm$-full for  $n\gg0$ and any ideal $K$.
\end{enumerate}
\end{cor}

\begin{proof}
Take $J=0$ in \ref{fullthm}.
\end{proof}

\begin{remark}
The notions in Definition \ref{fulldef} can be extended to submodules, see for instance \cite{HRR} or \cite[Section 8.4.3]{V}. One can use Theorem \ref{mainthm} to derive similar results to \ref{fullthm} and \ref{fullcor}.   

\end{remark}

The following observation involves Burch ideals, which turns out to be rather easy to construct by adding products with $\fm$. 

\begin{prop}
Let $(R,\fm,k)$ be a local ring and $I,J$ be ideals of $R$ such that $J\fm \nsubseteq I\fm$. Then $I+J\fm$ is Burch. In particular, if $\dim R/I>0$ and $\dim R/J=0$, then $I+J\fm$ is Burch. 

\end{prop}

\begin{proof}
Suppose $I+J\fm$ is not Burch, then $I\fm+J\fm^2 = [(I+J\fm):\fm]\fm \supseteq (I+J)\fm$. Working modulo $I\fm$, we get $\overline {J\fm^2} \subseteq \overline {J\fm}$, so $J\fm\subseteq I\fm$. The second assertion is clear. 
\end{proof}

Because of Theorem \ref{fullthm} and Corollary \ref{fullcor}, it seems reasonable to make:

\begin{definition}\label{ndef}
For an ideal $I$ in a local ring $(R, \fm,k)$ with $\depth R>0$, we define:
$$n_1(I) := \min \{t \geq 0 \ | \ I\fm^n \ \text{is $\fm$-full for all}\  n\geq t \}$$
$$n_2(I) := \min \{t \geq 0 \ | \ I\fm^n \ \text{is  full for all}\  n\geq t \}$$
$$n_3(I) := \min \{t \geq 0 \ | \ I\fm^n \ \text{is weakly $\fm$-full for all}\  n\geq t \}$$

\end{definition}

\begin{remark}
Clearly $n_1(I)\geq \max\{n_2(I), n_3(I)\}$.
\end{remark}

These invariants will be shown to be equal when $R$ is a regular local ring of dimension $2$, see the next section. 

The following observation comes from a question by Neil Epstein.
\begin{prop}
Let $(P)$ be one of the properties $\{$$\fm$ -full, weakly $\fm$-full, full$\}$. Let $\{I_i\}_{i\in X}$ be a family of ideals such that each $I_i$ is $(P)$. Then $\bigcap_{i\in X}I_i$ is $(P)$ (for $\fm$-full or full we need to assume that the cardinality of $X$ is less than that of the residue field $k$). 
\end{prop}

\begin{proof}
Suppose each $I_i$ is $\fm$-full and let $U_i\subseteq V=\fm/\fm^2$ be the Zariski open set for which the condition $\fm I_i:x=I_i$ holds when the image of $x$ is in $U_i$.  We claim that  $\bigcap_{i\in X} U_i$ is non-empty (this is where we need the cardinality condition). Let $V_i=V-U_i$, then each $V_i$ has dimension less than $\dim V$. If $\dim V=1$, then $\bigcup V_i$ has cardinality $|X|$, while $|V| = |k|$, so we are done. If $\dim V>1$, one can do induction by taking a general hyperplane $H$ such that  $\dim V_i\cap H<\dim V_i$ for each $i$.

By the above claim, for a general $x$, $\fm(\bigcap_{i\in X}I_i):x \subseteq \fm I_i:x=I_i$ for each $i\in X$. Thus the left hand side is in $\bigcap_{i\in X}I_i$ and we are done. 

For full ideals, we use the existence of general $x$ as above and $(\bigcap_{i\in X}I_i): J = \bigcap_{i\in X}(I_i:J)$. 
The proof for weakly $\fm$-full is simpler as we don't need to use cardinalities. 
\end{proof}

\section{Two dimensional regular local rings}

In this section we focus on the case when $R$ is a regular local ring of dimension two. 
In this case, any ideal $I$ can be written as $I=fJ$ where $J$ is $\fm$-primary, and it is easy to see that $I:x=f(J:x)$ and $I:\fm=f(J:\fm)$. Thus, using the results on $\fm$-primary ideals carefully developed in \cite[Chapter 14]{HS}, we see that: 

\begin{prop}\label{fullfacts}
Let $(R,\fm,k)$ be a regular local ring of dimension two and $I$ be an ideal. Write $I=fJ$ where $J$ is $\fm$-primary. The following are equivalent:
\begin{enumerate}
\item $I$ is $\fm$-full. 
\item $I$ is full. 
\item $J$ is $\fm$-full. 
\item $J$ is full. 
\item $\mu(J) = \ord(J)+1$.
\end{enumerate}
\end{prop}

A crucial and interesting result in dimension two is that the product of two full ideals is full. However, even in this situation, the sum of two full ideals may not be full. For instance, take $I=(x^2), J=(y^2)$ or $I=(x^2,xy^2,y^3)$ and $J=(x^3,x^2y,y^2)$. We shall establish a precise condition for when the sum of two full ideal is full. 

\begin{prop}\label{dim2colon}
Let $(R,\fm,k)$ be a regular local ring of dimension two and $I,J$ are nonzero ideals. The following are equivalent. 
\begin{enumerate}
\item $(I+J):x = (I:x)+(J:x)$ for a general $x\in \fm$.
\item $\ord(I\cap J) =\max\{\ord(I),\ord(J)\}$.
\end{enumerate}

\end{prop}

\begin{proof}
Consider the exact sequence $\ses{\frac{R}{I\cap J}}{\frac{R}{I}\oplus \frac{R}{J}}{\frac{R}{I+J}}$. Then quite generally, (1) is equivalent to the surjectivity of the map induced when tensoring the sequence with $R/xR$:
$$\Tor_1^R(R/I\oplus R/J, R/xR) = (I:x)/I\oplus (J:x)/J \to \Tor_1^R(R/(I+J), R/xR) = (I+J):x/(I+J) $$
Hence, (1) is equivalent to the exactness of $$\ses{\frac{R}{(I\cap J,x)}}{\frac{R}{(I,x)}\oplus \frac{R}{(J,x)}}{\frac{R}{(I+J,x)}}$$
Note that since we are in dimension two and $R$ is regular, $\len(R/(I,x)) = \ord(I)$ and $\ord(I+J) = \min \{\ord(I), \ord(J)\}$, so we are done. 
\end{proof}

\begin{cor}\label{dim2cor}
Let $I,J$ be full ideals such that $\ord(I\cap J) =\max\{\ord(I),\ord(J)\}$. Then $I+J$ is full. In particular $I+\fm^a$ is full for any $a$. 
\end{cor}

\begin{proof}
Then by  \ref{dim2colon}, for a general $x$:
$$(I+J):x = (I:x) +(J:x) \subseteq (I:\fm)+(J:\fm) \subseteq (I+J):\fm $$
which is all we need.  

For the last assertion, we need to show that $\ord(I\cap \fm^a) = \max \{\ord(I),a\}$. If $\ord(I)\geq a$ then $I\cap \fm^a =I$. If $b=\ord(I)< a$, then $I\fm^{a-b}\subseteq I\cap \fm^a \subseteq\fm^a$  which forces the desired equality. 

\end{proof}

Finally, we prove the equality of the  invariants defined in \ref{ndef} in this special case. 

\begin{prop}
If $R$ is regular local of dimension $2$, then $n_1(I)= n_2(I)=n_3(I)$ for each ideal $I$. 
\end{prop}

\begin{proof}
Since being $\fm$-full and full are equivalent in this case, it suffices to prove $n_1(I)=n_3(I)$. Let $a=n_1(I)$ and $b=n_3(I)$. Since being weakly $\fm$-full is equivalent to $\mu(\fm I)=\mu(I)+1$. we have that $\mu(I\fm^b)=\mu(I\fm^a)+b-a$. However, as $I\fm^b$ is full, we have $\mu(I\fm^b)=\ord(I\fm^b)+1 = \ord(I\fm^a)+b-a+1$. So $\mu(I\fm^a)=  \ord(I\fm^a)+1$, showing that $I\fm^a$ is full already. 
\end{proof}

\begin{example}
Let $I=(x^a,y^a)\subset R=k[[x,y]]$. It is easy to show (e.g. using \ref{fullfacts}) that $I\fm^n$ is $\fm$-full if and only if $n\geq a-1$, so $n_i(I)=a-1$ for all $i\in \{1,2,3\}$.  
\end{example}

\begin{question}
Can we find good lower and upper bounds for $n_i(I)$? Even when $R$ is regular and $2$-dimensional, it is not clear to the author how to do this. 
\end{question}

\end{document}